\documentclass[a4paper,12pt]{article} 
\usepackage{graphicx}
\usepackage[english]{babel} 
\usepackage{amsmath,amssymb,amsthm,amsfonts} 
\usepackage{latexsym}
\usepackage[all]{xy}
\usepackage{url}
\usepackage{tabularx}

\newcommand{\Q}{\mathbb Q}
\newcommand{\N}{\mathbb N}
\newcommand{\Z}{\mathbb Z}

\renewcommand{\P}{\mathfrak P}
\newcommand{\qq}{\mathfrak Q}
\newcommand{\p}{\mathfrak p}
\newcommand{\q}{\mathfrak q}
\newcommand{\gal}{\mathrm{Gal}}
\renewcommand{\epsilon}{\varepsilon}
\newcommand{\st}{\mathrm{st}}

\newcommand{\rt}{\mathrm{R}_t}
\newcommand{\cl}{\mathrm{Cl}}
\newcommand{\aut}{\mathrm{Aut}}
\newcommand{\ind}{\mathrm{ind}}
\newcommand{\res}{\mathrm{res}}

\newcommand{\G}{\mathcal G}
\newcommand{\oo}{\mathcal O}
\newcommand{\W}{\mathcal W}

\newtheorem{theorem}{Theorem}[section]
\newtheorem{lemma}[theorem]{Lemma}

\newtheorem{proposition}[theorem]{Proposition}

\theoremstyle{remark}
\newtheorem{remark}[theorem]{Remark}
\newtheorem{example}[theorem]{Example}
\newtheorem{conjecture}[theorem]{Conjecture}
\newtheorem{question}[theorem]{Question}

\theoremstyle{definition}
\newtheorem{definition}[theorem]{Definition}

\title{An explicit candidate for the set of Steinitz classes of tame Galois extensions with fixed Galois group of odd order}
\author{Luca Caputo and Alessandro Cobbe}

\begin{document}
\maketitle
\noindent\textsc{Abstract -}\begin{small}{Given a finite group $G$ and a number field $k$, a well-known conjecture asserts that the set $\rt(k,G)$ of Steinitz classes of tame $G$-Galois extensions of $k$ is a subgroup of the ideal class group of $k$. In this paper we investigate an explicit candidate for $\rt(k,G)$, when $G$ is of \emph{odd} order. More precisely, we define a subgroup $\mathcal{W}(k,G)$ of the class group of $k$ and we prove that $\rt(k,G)\subseteq \mathcal{W}(k,G)$. We show that equality holds for all groups of odd order for which a description of $\rt(k,G)$ is known so far. Furthermore, by refining techniques introduced in \cite{Cobbe1}, we use the Shafarevich-Weil Theorem in cohomological class field theory, to construct some tame Galois extensions with given Steinitz class. In particular, this allows us to prove the equality $\rt(k,G)=\mathcal{W}(k,G)$ when $G$ is a group of order dividing $\ell^4$, where $\ell$ is an odd prime.}
\end{small}
\begin{section}{Introduction}
Let $K/k$ be an extension of number fields and let $\oo_K$ and $\oo_k$ be the rings of integers of $K$ and $k$ respectively. Then
\[\oo_K\cong \oo_k^{[K:k]-1}\oplus I\]
as $\oo_k$-modules, where $I$ is an ideal in $\oo_k$ which is determined by the extension $K/k$ up to principal ideals. The ideal class of $I$ is called the Steinitz class $\st(K/k)$ of the extension.

For a number field $k$, denote by $\cl(k)$ the class group of $k$. If $G$ is a finite group, one can consider the set of all Steinitz classes which arise from tame $G$-Galois extensions of $k$:
\[\rt(k,G)=\{x\in\cl(k):\ \exists\, K/k\text{ tame Galois, }\gal(K/k)\cong G, \st(K/k)=x\}.\]
There is no known description of $\rt(k,G)$ in general, but in all cases for which it is known it turns out to be a subgroup of the ideal class group of $k$. This leads to formulate the following conjecture.

\begin{conjecture}\label{conj1}
For any number field $k$ and any finite group $G$, $\rt(k,G)$ is a subgroup of $\cl(k)$.
\end{conjecture}

The conjecture is known to hold for abelian groups as a consequence of the results of the paper \cite{McCulloh_Crelle} by Leon McCulloh. An explicit description of $\rt(k,G)$ for abelian groups of odd order was first found by Lawrence Paul Endo in his unpublished PhD thesis \cite{Endo}, in which he also obtained some explicit results in the even case.

Concerning nonabelian groups, there are a lot of results for groups of some particular shape, however sometimes with restrictions on the base field $k$. A list of recent results includes \cite{Bruche}, \cite{BrucheSodaigui}, \cite{ByottGreitherSodaigui}, \cite{Carter}, \cite{CarterSodaigui_quaternionigeneralizzati}, \cite{Cobbe2}, \cite{Cobbe1}, \cite{Cobbe3}, \cite{GodinSodaigui_A4}, \cite{GodinSodaigui_ottaedri}, \cite{Massy},  \cite{SbeitySodaigui}, \cite{Sodaigui1}, \cite{Sodaigui2} and \cite{Soverchia}.

The aim of this paper is to develop some new ideas to improve the results of \cite{Cobbe1} and \cite{Cobbe3}. In Section \ref{second} we generalize \cite[Lemma 3.14]{Cobbe1}, obtaining some limitations on the prime ideals which can ramify in tame Galois extensions of a number field with Galois group $G$. In particular, this gives an inclusion of $\rt(k,G)$ in an explicitly described subgroup of $\cl(k)$, which we baptize $\W(k,G)$ (see Definition \ref{WkG}). It seems natural to ask whether this inclusion is always an equality for groups of odd order (see Question \ref{question2}), while some results contained in \cite{Endo} give counterexamples already for abelian groups of even order. Note, in particular, that a positive answer to Question \ref{question2} would have Conjecture \ref{conj1} as a consequence (for groups of odd order).

In Section \ref{cohomol} we extend the result of \cite[Proposition 3.13]{Cobbe1}, using a theorem by Shafarevich and Weil as a tool to construct Galois extensions of a given number field. The main idea is to use some cohomological information encoded in the reciprocity map of an abelian extension of a number field, which is itself a Galois extension of a smaller field.

In the last two sections we answer affirmatively to Question \ref{question2} for $A'$-groups of odd order ($A'$-groups have been introduced in \cite{Cobbe1}, see also Definition \ref{A'groups}) and for $\ell$-groups of order up to $\ell^4$, where $\ell$ is an odd prime.

The case of groups of prime power order already shows some interesting complications, since for instance a two-step approach corresponding to a given splitting of the group as a semidirect product is not always enough to obtain all the needed Steinitz classes. The difficulty here comes from those extensions with prime ideals ramifying in both steps which, contrary to the case of $A'$-groups, cannot be left out. More precisely, given an \emph{arbitrary} element $\tau\in G$, we need to exhibit tame $G$-Galois extensions such that $\tau$ is a generator of the inertia group of some prime ideals belonging to some prescribed ideal classes. To this aim, in Section \ref{lgruppi}, we develop a technique involving the construction of a larger Galois extension of the ground field, whose Galois group has $G$ as a quotient (see the proof of Lemma \ref{type3}). This method may be used to prove the equality $\rt(k,G)=\W(k,G)$ for a lot of other groups $G$ of some particular kind. Anyway, we do not believe that it is worth to include more examples and we feel that one probably needs significantly new ideas to get some more general results.

\end{section}

\begin{section}{Some necessary conditions for the ramification of prime ideals in tame Galois extensions}\label{second}
For the rest of the paper we will use the letter $G$ to denote a finite group. For any $\tau\in G$, let $N_G(\tau)$ (resp. $C_G(\tau)$) be the normalizer (resp. the centralizer) of the subgroup of $G$ generated by $\tau$ and let $o(\tau)$ be the order of $\tau$. For a positive integer $s$, $\zeta_s$ will denote a primitive $s$-th root of unity. Finally we will always use the letter $\ell$ only for rational prime numbers, even if not explicitly indicated.

For any $\tau\in G$, we define a homomorphism
\[\varphi_\tau:N_G(\tau)\to(\Z/o(\tau)\Z)^\times\]
by $\varphi_\tau(\sigma)=\alpha$, where $\sigma\tau\sigma^{-1}=\tau^\alpha$. Furthermore, if $k$ is a number field, we need the cyclotomic character
\[\nu_{k,\tau}:\gal(k(\zeta_{o(\tau)})/k)\to (\Z/o(\tau)\Z)^\times,\]
which is defined by $\nu_{k,\tau}(g)=\beta$, where $g(\zeta_{o(\tau)})=\zeta_{o(\tau)}^\beta$. Note that $\nu_{k,\tau}$ is injective. We also set
\[H_{k,G,\tau}=\nu_{k,\tau}^{-1}(\varphi_\tau(N_G(\tau)))\subseteq \gal({k(\zeta_{o(\tau)})}/k)\]
and we call $E_{k,G,\tau}$ the subfield of $k(\zeta_{o(\tau)})$ which is fixed by $H_{k,G,\tau}$. 

\begin{remark}
We have the following commutative diagram

\[\xymatrix{
\varphi_\tau^{-1}(\nu_{k,\tau}(\gal({k(\zeta_{o(\tau)})}/k)))\ar@{^{(}->}[rr]^-{\mathrm{\iota}}\ar[d]^{\nu_{k,\tau}^{-1}\circ\varphi_\tau}&&N_G(\tau)\ar[d]^{\varphi_\tau}\\ \gal({k(\zeta_{o(\tau)})}/k)\ar[rr]^-{\nu_{k,\tau}}&& (\Z/o(\tau)\Z)^\times}\]
and it is easy to verify that $\varphi_\tau^{-1}(\nu_{k,\tau}(\gal({k(\zeta_{o(\tau)})}/k)))$ is the pullback of $\varphi_\tau$ and $\nu_{k,\tau}$. Its image in $\gal({k(\zeta_{o(\tau)})}/k)$ is exactly $H_{k,G,\tau}$.
\end{remark}

We start proving some elementary properties of the number fields $E_{k,G,\tau}$, which will be useful throughout the paper.

\begin{lemma}\label{primolemmaE}
Suppose that $G=H\rtimes_\mu \G$ is the (internal) semidirect product of two subgroups, $H$ and $\G$. Then, for all $\sigma\in\G$, we have
\[E_{k,G,\sigma}=E_{k,\G,\sigma}.\]
If $H$ is abelian, then, for all $\tau\in H$, $E_{k,G,\tau}$ coincides with the $E_{k,\mu,\tau}$ defined in \cite[Section 3.1]{Cobbe1}.
\end{lemma}

\begin{proof}
Let $\sigma\in\G$ and let $\sigma_1\in\G$, $\tau_1\in H$ be such that $\sigma_1\tau_1\in N_G(\sigma)$. Then
\[\sigma^{\varphi_\sigma(\sigma_1\tau_1)}=\sigma_1\tau_1\sigma\tau_1^{-1}\sigma_1^{-1}=(\sigma_1\tau_1\sigma_1^{-1})(\sigma_1\sigma\tau_1^{-1}\sigma^{-1}\sigma_1^{-1})\sigma_1\sigma\sigma_1^{-1}.\]
Since $\sigma_1\tau_1\sigma_1^{-1}$, $\sigma_1\sigma\tau_1^{-1}\sigma^{-1}\sigma_1^{-1}\in H$ and each element of $G$ can be written uniquely as a product of an element of $H$ and one of $\G$, we deduce that
\[(\sigma_1\tau_1\sigma_1^{-1})(\sigma_1\sigma\tau_1^{-1}\sigma^{-1}\sigma_1^{-1})=1\]
and
\[\sigma^{\varphi_\sigma(\sigma_1\tau_1)}=\sigma_1\sigma\sigma_1^{-1}.\]
Hence $\sigma_1\in N_\G(\sigma)$ and $\varphi_\sigma(\sigma_1\tau_1)=\varphi_\sigma(\sigma_1)$. Therefore we obtain
\[\varphi_\sigma(N_G(\sigma))=\varphi_\sigma(N_\G(\sigma)),\]
i.e.
\[E_{k,G,\sigma}=E_{k,\G,\sigma}.\]
The second part of the lemma is trivial.
\end{proof}

\begin{lemma}\label{inclEpotsigma}\label{generalostesso}
Let $\sigma\in G$ and $n\in\N$. Then
\[E_{k,G,\sigma^n}\subseteq E_{k,G,\sigma}.\]
In particular if $\tau\in G$ generates the same cyclic subgroup as $\sigma$, then
\[E_{k,G,\sigma}=E_{k,G,\tau}.\]
\end{lemma} 

\begin{proof}
We consider the restriction
\[\res:\gal({k(\zeta_{o(\sigma)})}/k)\to\gal({k(\zeta_{o(\sigma^n)})}/k)\]
and the projection
\[\pi:(\Z/o(\sigma)\Z)^\times\to (\Z/o(\sigma^n)\Z)^\times.\]
Using that $N_G(\sigma)\subseteq N_G(\sigma^n)$, we get commutative diagrams
\[\xymatrix{
N_G(\sigma)\ar[r]^-{\varphi_\sigma}\ar@{^{(}->}[d]^{\iota}&(\Z/o(\sigma)\Z)^\times\ar[d]^{\pi}\\ N_G(\sigma^n)\ar[r]^-{\varphi_{\sigma^n}}& (\Z/o(\sigma^n)\Z)^\times}
\qquad\qquad
\xymatrix{
\gal({k(\zeta_{o(\sigma)})}/k)\ar[r]^-{\nu_{k,\sigma}}\ar[d]^{\res}&(\Z/o(\sigma)\Z)^\times\ar[d]^{\pi}\\ \gal({k(\zeta_{o(\sigma^n)})}/k)\ar[r]^-{\nu_{k,\sigma^n}}& (\Z/o(\sigma^n)\Z)^\times.}
\]
Therefore $H_{k,G,\sigma}\subseteq \res^{-1}(H_{k,G,\sigma^n})$, since
\[\begin{split}
H_{k,G,\sigma}&=\nu_{k,\sigma}^{-1}(\varphi_\sigma(N_G(\sigma)))\\
&\subseteq \res^{-1}(\nu_{k,\sigma^n}^{-1}(\nu_{k,\sigma^n}( \res(\nu_{k,\sigma}^{-1}(\varphi_\sigma(N_G(\sigma)))))))\\
&= \res^{-1}(\nu_{k,\sigma^n}^{-1}(\pi(\nu_{k,\sigma}(\nu_{k,\sigma}^{-1}(\varphi_\sigma(N_G(\sigma)))))))\\
&\subseteq \res^{-1}(\nu_{k,\sigma^n}^{-1}( \pi(\varphi_\sigma(N_G(\sigma)))))\\
&=\res^{-1}( \nu_{k,\sigma^n}^{-1}(\varphi_{\sigma^n}(\iota(N_G(\sigma)))))\\
&\subseteq\res^{-1}( \nu_{k,\sigma^n}^{-1}(\varphi_{\sigma^n}(N_G(\sigma^n))))=\res^{-1}(H_{k,G,\sigma^n}).
\end{split}\]
Now the fixed field of $\res^{-1}(H_{k,G,\sigma^n})$ in $k(\zeta_{o(\sigma)})$ coincides with the fixed field of $H_{k,G,\sigma^n}$ in $k(\zeta_{o(\sigma^n)})$, i.e. with $E_{k,G,\sigma^n}$.
Therefore
\[E_{k,G,\sigma^n}\subseteq E_{k,G,\sigma}.\]
\end{proof}

\begin{lemma}\label{coniugati}
Let $\sigma,\tau\in G$, then \[E_{k,G,\sigma}=E_{k,G,\tau\sigma\tau^{-1}}.\]
\end{lemma}

\begin{proof}
We observe that $N_G(\tau\sigma\tau^{-1})=\tau N_G(\sigma)\tau^{-1}$. Further, for any $\rho\in N_G(\sigma)$, we have
\[\begin{split}(\tau\sigma\tau^{-1})^{\varphi_{\tau\sigma\tau^{-1}}(\tau\rho\tau^{-1})}&=(\tau\rho\tau^{-1})(\tau\sigma\tau^{-1})(\tau\rho\tau^{-1})^{-1}\\&=\tau\rho\sigma\rho^{-1}\tau^{-1}=\tau\sigma^{\varphi_\sigma(\rho)}\tau^{-1}=(\tau\sigma\tau^{-1})^{\varphi_\sigma(\rho)},\end{split}\]
i.e.
\[\varphi_{\tau\sigma\tau^{-1}}(\tau\rho\tau^{-1})=\varphi_\sigma(\rho).\]
Therefore
\[\varphi_{\tau\sigma\tau^{-1}}(N_G(\tau\sigma\tau^{-1}))=\varphi_{\tau\sigma\tau^{-1}}(\tau N_G(\sigma)\tau^{-1})=\varphi_\sigma(N_G(\sigma))\]
and, since clearly $\nu_{k,\sigma}=\nu_{k,\tau\sigma\tau^{-1}}$, we conclude that $H_{k,G,\sigma}=H_{k,G,\tau\sigma\tau^{-1}}$, and therefore $E_{k,G,\sigma}=E_{k,G,\tau\sigma\tau^{-1}}$.
\end{proof}

We recall the following important definition from \cite[Section I.2]{Endo} (see also \cite[Definition 2.8 and Proposition 2.10]{Cobbe1}).
\begin{definition}\label{defW}
For any finite abelian extension $K/k$ of number fields, we define
\[W(k,K)=N_{K/k}J_K\cdot P_k/P_k\subseteq \cl(k),\]
where $J_K$ is the group of fractional ideals of $K$, $P_k$ is the group of principal ideals in $k$ and the map $N_{K/k}$ is induced by the norm from $K$ to $k$. For any positive integer $m$, we use the notation $W(k,m)=W(k,k(\zeta_m))$.
\end{definition}

The following lemma is an immediate consequence of the definition.

\begin{lemma}\label{Wrovescia}
Let $L\supseteq K\supseteq k$ be number fields. Then
\[W(k,L)\subseteq W(k,K).\]
\end{lemma}
\begin{proof}
Actually: $N_{L/k}J_L=N_{K/k} (N_{L/K}J_L)\subseteq N_{K/k}J_K$.
\end{proof}

For an integer $n$ and a prime number $\ell$, we will always indicate by $n(\ell)$ the greatest power of $\ell$ dividing $n$. For an element $\sigma\in G$ of order $o(\sigma)$, we will denote by $\sigma(\ell)$ the element $\sigma^{o(\sigma)/o(\sigma)(\ell)}$, which is of order $o(\sigma)(\ell)$. Further we will denote by $G\{\ell\}$ the set of all the elements $\sigma\in G$ such that $\sigma(\ell)=\sigma$, i.e. of all the elements of order a power of $\ell$.

Let $A$ be a subgroup of an abelian group $B$ (written multiplicatively). For $t\in \frac{1}{2}\N\subset \mathbb{Q}$, we set 
\[A_B^t=\begin{cases}
\{x\in A:\ \exists\, a\in A,\ x=a^t\}&\textrm{if $t\in\N$}\\
\{x\in B:\ \exists\, a\in A,\ x^2=a^{2t}\}&\textrm{if $t\not\in\N$}.
\end{cases}.\]

In the above notation, we shall always omit the index $B$, since it will be clear from the context (actually $B$ will be an ideal class group).

We explicitly remark that in general $A\neq (A^2)^{1/2}$. For example, if $A$ is the trivial subgroup of the group $B$ with $2$ elements, then $(A^2)^{1/2}=B\neq A$.

\begin{lemma}\label{2.6acta}
Let $m$, $n$ be positive integers. If every prime $q$ dividing $n$ also divides $m$, then $W(k,m)^n\subseteq W(k,mn)$.
\end{lemma}

\begin{proof}
By the assumptions on $m$ and $n$, we have that \[[\Q(\zeta_{mn}):\Q(\zeta_{m})]=\frac{\phi(mn)}{\phi(m)}=n,\]
and this is a multiple of $[k(\zeta_{mn}):k(\zeta_{m})]$. Now, for any $x\in W(k,m)$, there exists an ideal $I$ in $k(\zeta_{m})$ such that $N_{k(\zeta_{m})/k}I$ is in the class of $x$. Hence we can conclude that $x^n\in W(k,mn)$, since
\[(N_{k(\zeta_{m})/k}I)^n=N_{k(\zeta_{mn})/k}(I^{n/[k(\zeta_{mn}):k(\zeta_{m})]}\oo_{k(\zeta_{mn})}).\]
The inclusion $W(k,m)^n\subseteq W(k,mn)$ follows.
\end{proof}

\begin{lemma}\label{An2inclusione}
Let $A$ be a subgroup of an abelian group $B$. If $d|m$, then $A^{m/2}\subseteq A^{d/2}$.
\end{lemma}
\begin{proof}
Let $x\in A^{m/2}$. If $d$ is even, then also $m$ must be even, and it is clear that $x\in A^{d/2}$. If $d$ is odd, then $x^2\in A^m\subseteq A^d$, i.e. $x\in A^{d/2}$.
\end{proof}

\begin{lemma}\label{Agcd}
Let $A$ be a subgroup of an abelian group $B$, let $m_1,\dots,m_n$ be positive integers and let $d$ be their greatest common divisor. Then
\[\prod_{i=1}^n A^{m_i/2}=A^{d/2}.\]
\end{lemma}
\begin{proof}
\begin{enumerate}
\item[$\subseteq$] By Lemma \ref{An2inclusione}, we have $A^{m_i/2}\subseteq A^{d/2}$ for $i=1,\dots,n$.
\item[$\supseteq$] Let $x\in A^{d/2}$ and let $c_1,\dots,c_n\in\Z$ be such that $\sum_{i=1}^n c_im_i=d$. If $d$ is even, then also every $m_i$ is even, and there exists $a\in A$ such that
\[x=a^{d/2}=a^{\sum_{i=1}^n c_im_i/2}=\prod_{i=1}^n(a^{c_i})^{m_i/2}\in \prod_{i=1}^n A^{m_i/2}.\]

If $d$ is odd, then there exists $a\in A$ such that $x^2=a^d$. Now let $i\in\{1,\dots,n\}$. If $m_i$ is odd, then
\[(x^{c_im_i/d})^2=a^{d(c_im_i/d)}=a^{c_im_i},\]
i.e.
\[x^{c_im_i/d}\in A^{m_i/2}.\]
If $m_i$ is even, then
\[x^{c_im_i/d}=x^{2c_im_i/(2d)}=a^{dc_im_i/(2d)}=a^{c_im_i/2}\in A^{m_i/2}.\]
Hence
\[x=x^{\sum_{i=1}^n c_im_i/d}=\prod_{i=1}^n x^{c_im_i/d}\in \prod_{i=1}^n A^{m_i/2}.\]
\end{enumerate}
\end{proof}
If $G$ is a group and $S$ is a subset of $G$ containing $1$, we use the notation $S^*$ for the subset $S\setminus\{1\}$.

\begin{proposition}\label{equivprimisep}
Let $k$ be a number field and let $G$ be a finite group. Then, for $i=0$ or $1$, we have
\[\prod_{\tau\in G^*}W(k,E_{k,G,\tau})^{\frac{o(\tau)-1}{2^i}\frac{\#G}{o(\tau)}}=\prod_{\ell|\#G}\ \prod_{\sigma\in G\{\ell\}^*}W(k,E_{k,G,\sigma})^{\frac{\ell-1}{2^i}\frac{\#G}{o(\sigma)}}.\]
\end{proposition}

\begin{proof}
\begin{enumerate}
\item[$\subseteq$]
Using \cite[Lemma 3.16]{Cobbe1}, Lemma \ref{An2inclusione} and Lemma \ref{Agcd} (for the first inclusion in the following formula), Lemma \ref{inclEpotsigma} and Lemma \ref{Wrovescia} (for the second inclusion in the following formula), we see that
\[\begin{split}W(k,E_{k,G,\tau})^{\frac{o(\tau)-1}{2^i}\frac{\#G}{o(\tau)}}&\subseteq \prod_{\ell|o(\tau)}W(k,E_{k,G,\tau})^{\frac{\ell-1}{2^i}\frac{\#G}{o(\tau)(\ell)}}\\&\subseteq \prod_{\ell|o(\tau)}W(k,E_{k,G,\tau(\ell)})^{\frac{\ell-1}{2^i}\frac{\#G}{o(\tau(\ell))}}\\&\subseteq \prod_{\ell|\#G}\ \prod_{\sigma\in G\{\ell\}^*} W(k,E_{k,G,\sigma})^{\frac{\ell-1}{2^i}\frac{\#G}{o(\sigma)}}.\end{split}\]
\item[$\supseteq$] Let $\sigma\in G\{\ell\}^*$, where $\ell$ is a prime number dividing $\#G$, and set $\tilde\sigma=\sigma^{o(\sigma)/\ell}$.
Clearly
\begin{equation}\gcd\left((\ell-1)\frac{\#G}{\ell},(o(\sigma)-1)\frac{\#G}{o(\sigma)}\right)=(\ell-1)\frac{\#G}{o(\sigma)}.\label{gcd}\end{equation}
Hence, using Lemma \ref{inclEpotsigma}, Lemma \ref{Wrovescia} and Lemma \ref{Agcd},
\[\begin{split}W(k,E_{k,G,\sigma})^{\frac{\ell-1}{2^i}\frac{\#G}{o(\sigma)}}&= W(k,E_{k,G,\sigma})^{\frac{\ell-1}{2^i}\frac{\#G}{\ell}}W(k,E_{k,G,\sigma})^{\frac{o(\sigma)-1}{2^i}\frac{\#G}{o(\sigma)}}\\
&\subseteq W(k,E_{k,G,\tilde\sigma})^{\frac{o(\tilde\sigma)-1}{2^i}\frac{\#G}{o(\tilde\sigma)}}W(k,E_{k,G,\sigma})^{\frac{o(\sigma)-1}{2^i}\frac{\#G}{o(\sigma)}}\\
&\subseteq \prod_{\tau\in G^*}W(k,E_{k,G,\tau})^{\frac{o(\tau)-1}{2^i}\frac{\#G}{o(\tau)}}.
\end{split}\]
\end{enumerate}
\end{proof}

\begin{definition}\label{WkG}
Let $k$ be a number field and $G$ be a finite group. Then we set
\[\W(k,G)=\prod_{\tau\in G^*}W(k,E_{k,G,\tau})^{\frac{o(\tau)-1}{2}\frac{\#G}{o(\tau)}}=\prod_{\ell|\#G}\ \prod_{\sigma\in G\{\ell\}^*}W(k,E_{k,G,\sigma})^{\frac{\ell-1}{2}\frac{\#G}{o(\sigma)}}\subseteq\cl(k).\]
\end{definition}

We now come to the principal results of this section. We shall show that the classes of primes ramifying in a tame $G$-Galois extension of $k$ have to satisfy some conditions. As a consequence, we will get the inclusion $\rt(k,G)\subseteq \W(k,G)$.

If $K/k$ is a Galois extension of number fields and $\P$ is a prime ideal in $K$, we denote by $D(K/k,\P)$ (resp. $I(K/k,\P)$) the decomposition group (resp. the inertia group) of $\P$ in $K/k$. Further we denote by $d(K/k)$ the discriminant of $K$ over $k$.

\begin{proposition}\label{primiramificati}
Let $K/k$ be a tame $G$-Galois extension of number fields, let $\p$ be a prime of $k$ ramifying in $K/k$ and let $\tau$ be a generator of $I(K/k,\P)$, where $\P$ is a prime of $K$ dividing $\p$. Then
\[x\in W(k,E_{k,G,\tau}),\]
where $x$ is the class of $\p$ in $\cl(k)$.
\end{proposition}

\begin{proof}
Let $\tilde{\P}$ be a prime above $\P$ in $K(\zeta_{o(\tau)})$. In order to show that $x\in W(k,E_{k,G,\tau})$, we will prove that $D(K(\zeta_{o(\tau)})/k,\tilde{\P})$ is contained in $\gal(K(\zeta_{o(\tau)})/E_{k,G,\tau})$. It will follow that $\p$ has inertia degree $1$ in $E_{k,G,\tau}/k$, and hence that it is the norm of an ideal in $E_{k,G,\tau}/k$.

Recall that there is a canonical isomorphism between $I(K/k,\P)$ and a subgroup of the units of the residue field $\kappa_\P$ of $K$ at $\P$, which is easily seen to be invariant with respect to the action of the decomposition group of $\P$ in $K/k$ (see \cite[Proposition IV.2.7]{SerreLocalFields}). Hence $\kappa_\P^{\times}$ contains an element of order $o(\tau)$ and therefore $\kappa_{\P}=\kappa_{\tilde{\P}}$, where the latter denotes the residue field of $K(\zeta_{o(\tau)})$ at $\tilde{\P}$.

Let now $\delta\in D(K(\zeta_{o(\tau)})/k,\tilde{\P})$. Then $\delta|_K\in N_G(\tau)$, since $\delta|_K\in D(K/k,\P)$ and $I(K/k,\P)\triangleleft D(K/k,\P)$. By definition, 
\[\tau^{\varphi_{\tau}(\delta|_K)}=(\delta|_K)\tau(\delta|_K)^{-1}\]
and therefore, using the invariance of the above mentioned isomorphism, $\delta|_K$ (and hence also $\delta$) acts as raising to the power $\varphi_{\tau}(\delta|_K)$ on the subgroup of order $o(\tau)$ of $\kappa_\P^\times=\kappa_{\tilde{\P}}^\times$. Recalling that the powers of $\zeta_{o(\tau)}$ are distinct modulo $\tilde\P$ (since $\tilde\P\nmid o(\tau)$), we get that 
$$\zeta_{o(\tau)}^{\nu_{k,\tau}(\delta|_{k(\zeta_{o(\tau)})})}=\delta(\zeta_{o(\tau)})=\zeta_{o(\tau)}^{\varphi_{\tau}(\delta|_K)}.$$
In particular, $\nu_{k,\tau}(\delta|_{k(\zeta_{o(\tau)})})=\varphi_{\tau}(\delta|_K)$, i.e. 
\[\delta|_{k(\zeta_{o(\tau)})}\in H_{k,G,\tau}=\gal(k(\zeta_{o(\tau)})/E_{k,G,\tau}).\]
Thus
\[\delta\in \gal(K(\zeta_{o(\tau)})/E_{k,G,\tau})\]
and, since the choice of $\delta\in D(K(\zeta_{o(\tau)})/k,\tilde{\P})$ was arbitrary, we have
\[D(K(\zeta_{o(\tau)})/k,\tilde{\P})\subseteq \gal(K(\zeta_{o(\tau)})/E_{k,G,\tau}).\]
This shows that $x\in W(k,E_{k,G,\tau})$.
\end{proof}

\begin{theorem}\label{secondinclusion}
Let $k$ be a number field and $G$ a finite group. Then
\[\rt(k,G)\subseteq\mathcal{W}(k,G).\]
\end{theorem}

\begin{proof}
Let $K/k$ be a tame Galois extension with Galois group $G$. Its discriminant is
\[d(K/k)=\prod_\p \p^{(e_\p-1)\frac{\#G}{e_\p}},\]
where the product runs over the primes $\p$ of $k$ which ramify in $K/k$ and whose ramification index in $K/k$ is denoted by $e_\p$. Note that $e_\p$ equals the order of an element $\tau$ generating the inertia group of a prime ideal $\P$ of $K$ dividing $\p$.

If $G$ has odd order or noncyclic $2$-Sylow subgroups, then $d(K/k)$ is the square of an ideal whose class is $\st(K/k)$, by \cite[Theorem 2.1]{Cobbe1}. Hence, in this case, using Proposition \ref{primiramificati} we obtain
\[\st(K/k)\in \prod_{\tau\in G^*}W(k,E_{k,G,\tau})^{\frac{o(\tau)-1}{2}\frac{\#G}{o(\tau)}}=\W(k,G).\]

If $G$ has nontrivial cyclic $2$-Sylow subgroups, then by \cite[Theorem 2.1]{Cobbe1} and Proposition \ref{primiramificati} we only obtain
\[\st(K/k)\in \left(\prod_{\tau\in G^*}W(k,E_{k,G,\tau})^{(o(\tau)-1)\frac{\#G}{o(\tau)}}\right)^{1/2}\]
i.e., using Proposition \ref{equivprimisep},
\[\st(K/k)^2\in \prod_{\ell|\#G}\ \prod_{\sigma\in G\{\ell\}^*}W(k,E_{k,G,\sigma})^{(\ell-1)\frac{\#G}{o(\sigma)}}.\]
This means that, for every $\ell|\#G$ and every $\sigma\in G\{\ell\}^*$, there exists $x_\sigma\in W(k,E_{k,G,\sigma})$ such that
\[\st(K/k)^2=\prod_{\ell|\#G}\ \prod_{\sigma\in G\{\ell\}^*}x_\sigma^{(\ell-1)\frac{\#G}{o(\sigma)}}.\]
In the above formula, the only odd exponents in the right-hand side are those corresponding to the elements $\sigma\in G$ of order $\#G(2)$ (see the notation after Lemma \ref{Wrovescia}). For all the other terms we have
\[\prod_{\ell|\#G}\prod_{\substack{\sigma\in G\{\ell\}^*\\o(\sigma)\neq \#G(2)}}x_\sigma^{\frac{\ell-1}{2}\frac{\#G}{o(\sigma)}}\in \prod_{\ell|\#G}\ \prod_{\sigma\in G\{\ell\}^*}W(k,E_{k,G,\sigma})^{	\frac{\ell-1}{2}\frac{\#G}{o(\sigma)}}.\]
Further
\[\left(\st(K/k)\prod_{\ell|\#G}\prod_{\substack{\sigma\in G\{\ell\}^*\\o(\sigma)\neq \#G(2)}}x_\sigma^{-\frac{\ell-1}{2}\frac{\#G}{o(\sigma)}}\right)^2\in\prod_{\sigma\in G:\ o(\sigma)=\#G(2)}W(k,E_{k,G,\sigma})^{\frac{\#G}{o(\sigma)}}.\]
By the Sylow Theorems, all the $2$-Sylow subgroups of $G$ are conjugate to a single cyclic group, generated, say, by an element $\tilde\sigma$. Hence, for every $\sigma\in G$ of order $\#G(2)$, there exists $\tau\in G$ such that $\langle\sigma\rangle=\langle\tau\tilde\sigma\tau^{-1}\rangle$. Therefore, by Lemma \ref{coniugati} and Lemma \ref{generalostesso}, for all such $\sigma$, $E_{k,G,\sigma}=E_{k,G,\tilde\sigma}$. So we obtain
\[\st(K/k)\prod_{\ell|\#G}\prod_{\substack{\sigma\in G\{\ell\}^*\\o(\sigma)\neq \#G(2)}}x_\sigma^{-\frac{\ell-1}{2}\frac{\#G}{o(\sigma)}}\in W(k,E_{k,G,\tilde\sigma})^{\frac{1}{2}\frac{\#G}{o(\tilde\sigma)}}.\]
i.e.
\[\st(K/k)\in \prod_{\ell|\#G}\ \prod_{\sigma\in G\{\ell\}^*}W(k,E_{k,G,\sigma})^{\frac{\ell-1}{2}\frac{\#G}{o(\sigma)}}=\W(k,G).\]
This completes the proof.
\end{proof}

\begin{remark}
Using the results of L. McCulloh on realizable classes, one can obtain a different proof of Theorem \ref{secondinclusion}. More precisely for a tame $G$-Galois extension $K/k$, let $\mathrm{cl}_{\mathcal{O}_k[G]}(\mathcal{O}_K)$ denote the class defined by $\mathcal{O}_K$ in the projective class group $\cl(\mathcal{O}_k[G])$. Then McCulloh proved that, for every number field $k$ and every group $G$, the set  
\[R(\mathcal{O}_k[G])=\{\xi\in \cl(\mathcal{O}_k[G]):\ \exists\, K/k\text{ tame Galois, }\gal(K/k)\cong G, \mathrm{cl}_{\mathcal{O}_k[G]}(\mathcal{O}_K)=\xi\}\]
is contained in the kernel of a homomorphism defined on $\cl(\mathcal{O}_k[G])$ (see \cite{McCullohun1}). He then used this inclusion to derive an equivalent version of Theorem \ref{secondinclusion} (see \cite{McCullohun2}).
\end{remark}

One may wonder for which groups the inclusion of the above theorem is actually an equality (for every number field $k$). It is easy to find groups of even order for which the inclusion may be strict for some number fields.

\begin{example}
Let $k=\Q(i,\sqrt{10})$ and $G=C(8)$ be a cyclic group of order $8$. The field $k(\zeta_8)$ is obtained extending $k$ with a root of the polynomial $x^2-i$, whose roots are $\zeta_8$ and $\zeta_8^5$. Hence
\[\gal(k(\zeta_{8})/k)\subseteq \gal(\Q(\zeta_{8})/\Q)=(\Z/8\Z)^\times\]
is of order $2$ and is generated by the class of $5$. In particular it is different from the subgroups generated by the class of $-5$ and by the class of $-25$. We obtain by \cite[II.2.6]{Endo} that
\[\rt(k,C(8))=W(k,8).\]
On the other hand, it is easy to verify, using Lemma \ref{2.6acta}, that $\W(k,C(8))=W(k,8)^{\frac{1}{2}}$. We now show that $\rt(k,C(8))\neq\W(k,C(8))$. Actually, an easy calculation with PARI/GP shows that $\cl(k(\zeta_{8}))$ is trivial, while $\cl(k)$ is of order $2$. It follows that $\rt(k,C(8))=W(k,8)$ is the trivial group, while $\W(k,C(8))=W(k,8)^{\frac{1}{2}}$ is equal to $\cl(k)$, which is of order $2$.\qed
\end{example}

On the other hand, for a lot of groups of odd order, including those for which Conjecture \ref{conj1} is known to hold so far, we will show in Sections \ref{section4} and \ref{section5} that the inclusion of Theorem \ref{secondinclusion} is indeed an equality (for every number field $k$).
To simplify statements in the next sections, we will adopt the following working definitions.

\begin{definition}
Given a number field $F$, we will say that an extension $K/k$ of number fields satisfies the property $P(F)$ if it is tame, with no nontrivial unramified subextensions and such that $d(K/k)$ is prime to $d(F/\Q)$. To simplify notations we will say that $K/k$ satisfies $P(s)$ if it satisfies $P(k(\zeta_s))$.
\end{definition}

\begin{definition}\label{verygood}
We call a finite group $G$ \emph{very good} if the following conditions hold:
\begin{enumerate}
\item For every number field $k$, we have
\[\rt(k,G)=\W(k,G).\]
\item For every number field $k$, for every class $x\in\rt(k,G)$ and every positive integer $s$, there exists a $G$-Galois extension $K/k$ with Steinitz class $x$ and satisfying $P(s)$.
\end{enumerate}
\end{definition}

It is clear that a very good group of odd order is also good, in the sense of \cite[Definition 3.15]{Cobbe1}. Further it seems reasonable to expect that the second condition of the above definition is satisfied for every group $G$, including those not satisfying the first condition.

With these definitions and in view of the results that will be proved in Sections \ref{section4} and \ref{section5}, it seems natural to ask the following question.

\begin{question}\label{question2}
Is every group $G$ of odd order very good?
\end{question}

Of course a positive answer would also imply Conjecture \ref{conj1} for all groups of odd order.

We do not know the answer to Question \ref{question2}, but we suspect that it might be affirmative. The main point here is to construct tame $G$-Galois extensions with some given Steinitz classes; of course this could be very difficult in general, but in the next section we will give some results in this direction. 

\end{section}

\begin{section}{Constructive inclusion}\label{cohomol}
In this section we develop some techniques, which we will need in the last two sections to construct some tame Galois extensions of number fields with given Galois groups and Steinitz classes.

Let $\G$ be a finite group and let $H$ be an abelian group. Let $\mu:\G\to\aut(H)$ be an action of $\G$ on $H$ and let $G$ be a group such that there exists an exact sequence
\begin{equation}1\to H\to G\to \G\to 1\label{succesatta}\end{equation}
and the action of $\G$ on $H$ induced by the conjugation coincides with $\mu$. Let $k_1$ be a $\G$-Galois extension of a number field $k$ and $K/k_1$ an $H$-Galois extension such that $K/k$ is Galois with Galois group $G$. Let $C_{k_1}$ be the id\`ele class group of $k_1$.

\begin{theorem}[Shafarevich-Weil]\label{ShafarevichWeil}
With notation and hypotheses as above, the reciprocity map $(\ ,K/k_1):C_{k_1}\to H$ induces a map 
\[H^2(\G,C_{k_1})\to H^2(\G,H),\]
which sends the fundamental class $u\in H^2(\G,C_{k_1})$ to the class of the group extension (\ref{succesatta}). In particular, the image of $u$ is trivial precisely when $G$ is isomorphic to the semidirect product $H\rtimes_\mu \G$.
\end{theorem}

\begin{proof}
See \cite[Theorem 2.110]{ParshinShafarevich}.
\end{proof}

Using Theorem \ref{ShafarevichWeil} we can improve the result of \cite[Proposition 3.13]{Cobbe1}. In that paper, the action $\mu$ was assumed to be such that $H^2(\G,H)$ is trivial. This assumption is satisfied in some particular cases (for example when the orders of $H$ and $\G$ are coprime, by a theorem of Schur-Zassenhaus), but of course not in general. We are going to use Theorem \ref{ShafarevichWeil} to study groups $G$ of the form $H\rtimes_\mu \G$, where $H^2(\G,H)$ is no more assumed to be trivial. The proof will work exactly as in \cite{Cobbe1}, except for some details, which are pointed out in what follows; for the unchanged parts we will refer the reader to \cite{Cobbe1}.
\\

From now on, in this section, $G$ will be of the form $H\rtimes_\mu \G$, where $H$ is an abelian group of \emph{odd} order.

\begin{lemma}\label{cobbia310}
Let $s$ be a positive integer and let $k_1$ be a tame $\G$-Galois extension of $k$ satisfying $P(s)$. Then there exists a tame $H$-Galois extension $K$ of $k_1$, such that $K/k$ is $G$-Galois and satisfying the following conditions.
\begin{enumerate}
\item The prime ideals ramifying in $K/k_1$ are principal. 
\item $\st(K/k_1)=1$.
\item $K/k_1$ satisfies $P(s)$.
\end{enumerate}
\end{lemma}
\begin{proof}
Choose a set $\{\P_1,\dots,\P_t\}$ of prime ideals in $k_1$ which are unramified over $\Q$ and whose classes generate $\cl(k_1)$.
Now recall the definition of the so-called content map:
\[\pi:I_{k_1}\to J_{k_1},\qquad 
\alpha\mapsto\prod_{\p\nmid\infty}\p^{v_\p(\alpha_\p)},\]
where $I_{k_1}$ and $J_{k_1}$ are the id\`ele group and the ideal group of $k_1$, respectively. For every $i\in\{1,\dots,t\}$,
let $\pi_{\P_i}$ be a prime element in the completion $(k_1)_{\P_i}$ and let $y_i$ be the id\`ele in $I_{k_1}$ whose component at $\P_i$ (resp. at any prime different from $\P_i$) is $\pi_{\P_i}$ (resp. $1$).

Let $\tilde u:\G\times\G\to C_{k_1}$ be a cocycle, whose class in $H^2(\G,C_{k_1})$ is the fundamental class $u$. Recall that we have an exact sequence
\[1\to\prod_\P U_\P/U_{k_1}\to C_{k_1}\stackrel{\pi}{\rightarrow} \cl(k_1)\to 1,\]
where $U_{k_1}$ is the group of units of $k_1$ and $\prod_\P U_\P\subseteq I_{k_1}$ is the group of unit id\`eles.
Therefore, for all couples $(\rho_1,\rho_2)\in\G\times\G$, we can find $v_{\rho_1,\rho_2}\in\prod_\P U_\P$ and $\lambda_{\rho_1,\rho_2,j}\in\Z$ such that $\tilde u(\rho_1,\rho_2)$ is represented by $v_{\rho_1,\rho_2}\prod_{j=1}^t y_j^{\lambda_{\rho_1,\rho_2,j}}$.

We add all the elements $v_{\rho_1,\rho_2}$ to the set $\{u_1,\dots,u_T\}$ defined in \cite{Cobbe1}, before Lemma 3.9. The arguments of \cite[Lemma 3.10]{Cobbe1}, with $x=1$, give a tame $H$-Galois extension $K$ of $k_1$, which is also Galois over $k$, and such that the action of $\G$ on $H$ induced by the conjugation is $\mu$. Further, by construction, the kernel of the reciprocity map $(\ ,K/k_1):C_{k_1}\to \gal(K/k_1)=H$ contains all the elements $\tilde u(\rho_1,\rho_2)$ with $\rho_1,\rho_2\in\G$. Therefore the image of the fundamental class by the reciprocity map corresponding to the extension $K/k_1$ must be trivial and so, by Theorem \ref{ShafarevichWeil}, the Galois group of $K/k$ must be isomorphic to $H\rtimes_\mu \G$.

Finally, the construction of the proof of \cite[Lemma 3.10]{Cobbe1} implies that the primes $\q_1,\dots,\q_{r+1}$ of $k_1$ ramifying in $K/k_1$ are in the class of $x=1$, hence principal. Since $H$ is of odd order, we deduce that $\st(K/k_1)=1$, i.e. we have proved conditions 1 and 2. As for condition 3, it follows, once more by the proof of \cite[Lemma 3.10]{Cobbe1}, that $K/k_1$ has no nontrivial unramified subextensions and that the $\q_i$ can be chosen so that their restrictions to $\Q$ are unramified in $k_1(\zeta_s)/\Q$.

\end{proof}

\begin{lemma}\label{cobbia311}
Let $s$ be a multiple of $\#H$ and let $k_1$ be a tame $\G$-Galois extension of $k$ satisfying $P(s)$. For every prime number $\ell$ dividing $\#H$ and every $\tau\in H\{\ell\}^*$, let $x_\tau$ be any class in $W(k,E_{k,G,\tau})$ and let $A_\tau, B_\tau$ be nonnegative integers, with $A_\tau\neq1$. Then there exists a tame $H$-Galois extension $K$ of $k_1$, such that $K/k$ is $G$-Galois and satisfying the following conditions.
\begin{enumerate}
\item The only nonprincipal prime ideals of $k$ ramifying in $K/k_1$ are (for every $\ell$ dividing $\#H$ and every $\tau\in H\{\ell\}^*$):
\begin{enumerate}
\item[(a)]$\q_{\tau,1},\q_{\tau,2},\dots,\q_{\tau,A_\tau}$, in the class of $x_\tau$ and such that, for any $i=1,\dots,A_\tau$, the inertia group of every prime of $K$ dividing $\q_{\tau,i}$ is a conjugate of $\langle\tau^{\frac{o(\tau)}{\ell}}\rangle$;
\item[(b)]$\q_{\tau,A_\tau+1},\q_{\tau,A_\tau+2}$, in the class of $x_\tau^{B_\tau}$ and such that the inertia group of every prime of $K$ dividing $\q_{\tau,A_\tau+1}$ and $\q_{\tau,A_\tau+2}$ is a conjugate of $\langle\tau\rangle$.
\end{enumerate}
\item We have 
\[\st(K/k_1)=\prod_{\ell|\#H}\prod_{\tau\in H\{\ell\}^*}\iota(x_\tau)^{A_\tau\frac{\ell-1}{2}\frac{\#H}{\ell}+B_\tau(o(\tau)-1)\frac{\#H}{o(\tau)}},\] where $\iota:\cl(k)\to\cl(k_1)$ is induced by the inclusion $k\subseteq k_1$.
\item $K/k_1$ satisfies $P(s)$.
\end{enumerate}
\end{lemma}

\begin{proof}
By Lemma \ref{cobbia310} there exists a tame $H$-Galois extension $\tilde K$ of $k_1$, such that $\tilde K/k$ is $G$-Galois, $\tilde K/k_1$ is ramified only at principal ideals and it satisfies $P(s)$. 

Since $k_1/k$ satisfies $P(s)$, by \cite[Lemma 3.4]{Cobbe1} and Lemma \ref{primolemmaE}, we have for every $\tau\in H^*$ that $W(k,E_{k,G,\tau})=W(k,Z_{k_1/k,\mu,\tau})$, where $Z_{k_1/k,\mu,\tau}$ has been defined in \cite{Cobbe1}. This observation will allow us to argue as in \cite[Lemma 3.11]{Cobbe1}, but replacing $W(k,Z_{k_1/k,\mu,\tau})$ with $W(k,E_{k,G,\tau})$.

As in the proof of Lemma \ref{cobbia310}, we add all the $v_{\rho_1,\rho_2}$ (defined in the proof of Lemma \ref{cobbia310}) to the set $\{u_1,\dots,u_T\}$ defined in \cite{Cobbe1}, before Lemma 3.9.
Now, for every prime number $\ell$ dividing $\#H$, for every $\tau\in H\{\ell\}^*$ and for every $i = 1,\ldots,A_\tau$ (resp. $i=A_{\tau+1},A_{\tau+2}$) we choose prime ideals $\q_{\tau,i}$ in the class of $x_\tau$ (resp. in the class of $x_\tau^{B_\tau}$). For any $i=1,\dots,A_\tau+2$, we also choose prime ideals $\qq_{\tau,i}$ in $k_1$ dividing the $\q_{\tau,i}$ and generators $g_{\qq_{\tau,i}}$ of the units of the residue fields $\kappa_{\qq_{\tau,i}}$. As in the proof of \cite[Lemma 3.11]{Cobbe1}, we can assume that the $\q_{\tau,i}$ are unramified in $k_1(\zeta_s)/\Q$ and in $\tilde K/k_1$.

For every prime number $\ell$ dividing $\#H$, for every $\tau\in H\{\ell\}^*$ and for every $i=1,\dots,A_\tau+2$, we define $\varphi_{\tau,i}:\kappa_{\qq_{\tau,i}}^\times\to H$ as follows. If $A_\tau$ is odd, we set
\[\varphi_{\tau,i}(g_{\qq_{\tau,i}})=
\begin{cases}\tau^\frac{o(\tau)}{\ell}&\text{if $i=1,2$}\\
\tau^{-\frac{2o(\tau)}{\ell}}&\text{if $i=3$}\\
\tau^{(-1)^i\frac{o(\tau)}{\ell}}&\text{if $4\leq i\leq A_\tau$}\\
\tau^{(-1)^i}&\text{if $i=A_\tau+1,A_\tau+2$}.
\end{cases}\]
If $A_\tau$ is even, we set 
\[\varphi_{\tau,i}(g_{\qq_{\tau,i}})=
\begin{cases}
\tau^{(-1)^i\frac{o(\tau)}{\ell}}&\text{if $1\leq i\leq A_\tau$}\\
\tau^{(-1)^i}&\text{if $i=A_\tau+1,A_\tau+2$}.
\end{cases}\]
As in the proof of \cite[Lemma 3.11]{Cobbe1}, using that $x_\tau\in W(k,E_{k,G,\tau})$, we can assume that the $\qq_{\tau,i}$ and the $g_{\qq_{\tau,i}}$ are such that the $\varphi_{\tau,i}$ are $D(k_1/k,\qq_{\tau,i})$-invariant.
Hence, for $i=1,\dots,A_\tau+2$, we extend $\varphi_{\tau,i}$ to a $\G$-invariant homomorphism
\[\tilde\varphi_{\tau,i}:\ind^{\G}_{D(k_1/k,\qq_{\tau,i})}\kappa_{\qq_{\tau,i}}^\times\cong\prod_{\delta\in\G/D(k_1/k,\qq_{\tau,i})}\kappa_{\delta(\qq_{\tau,i})}^\times\to H,\]
using the Frobenius reciprocity. We define $\varphi_0:\prod_{\P}\kappa_{\P}^\times\to H$, setting
\[\begin{cases}
\varphi_0|_{\kappa_{\delta(\qq_{\tau,i})}^\times}=\tilde\varphi_{\tau,i} &\text{for $\ell|\#H$, $\tau\in H\{\ell\}^*$, $i=1,\dots,A_\tau+2$  and $\delta\in\G$}\\
\varphi_0|_{\kappa_{\P}^\times}=1 &\text{for all the other primes}.
\end{cases}\]
As in the proof of \cite[Lemma 3.10]{Cobbe1}, we can assume that all the elements $u_j$ are in the kernel of $\varphi_0$, so that we can extend $\varphi_0$ to a $\G$-invariant homomorphism $\varphi:C_{k_1}\to H$, whose kernel contains a congruence subgroup of $C_{k_1}$. Therefore taking the product
\[\varphi\cdot (\ ,\tilde K/k_1):C_{k_1}\to H,\]
we get a $\G$-invariant homomorphism whose kernel contains a congruence subgroup and which is surjective, by the assumptions on the ideals $\q_{\tau,i}$.

By class field theory (see for example \cite[Theorem 2.3 and Proposition 2,4]{Cobbe1}), we obtain a tame $H$-Galois extension $K$ of $k_1$, which is Galois over $k$. The extension $K/k_1$ satisfies $P(s)$, by the assumptions on the prime ideals $\q_{\tau,i}$ and the fact that $\tilde K/k_1$ satisfies $P(s)$. As in the proof of Lemma \ref{cobbia310}, actually $\gal(K/k)\cong G$. All the requested properties hold, by construction, once more as in the proof of \cite[Lemma 3.11]{Cobbe1}.
\end{proof}

For any $\G$-Galois extension $k_1$ of $k$, we define $\rt(k_1,k,G)\subseteq \cl(k_1)$ as the set of those ideal classes of $k_1$ which are Steinitz classes of a tamely ramified $H$-Galois extension $K/k_1$, such that $K/k$ is $G$-Galois.

\begin{lemma}\label{cobbia312}
Let $s$ be a multiple of $\#H$ and let $k_1$ be a tame $\G$-Galois extension of $k$ satisfying $P(s)$. Then
\[\prod_{\ell|\#H}\prod_{\tau\in H\{\ell\}^*}\iota\left(W\left(k,E_{k,G,\tau}\right)\right)^{\frac{\ell-1}{2}\frac{\#H}{o(\tau)}}\subseteq \rt(k_1,k,G),\]
where $\iota:\cl(k)\to \cl(k_1)$ is induced by the inclusion $k\subseteq k_1$. Further any class in the left-hand side of the above inclusion is the Steinitz class of a tame $H$-Galois extension $K/k_1$ which satisfies $P(s)$ and such that $K/k$ is $G$-Galois.
\end{lemma}

\begin{proof}
Let
\[x=\prod_{\ell|\#H}\prod_{\tau\in H\{\ell\}^*}x_\tau^{\frac{\ell-1}{2}\frac{\#H}{o(\tau)}}\in \prod_{\ell|\#H}\prod_{\tau\in H\{\ell\}^*}W\left(k,E_{k,G,\tau}\right)^{\frac{\ell-1}{2}\frac{\#H}{o(\tau)}},\]
with $x_\tau\in W(k,E_{k,G,\tau})$.
Let $o(x_\tau)$ be the order of $x_\tau\in \cl(k)$. Since any prime $\ell$ dividing $\#H$ is odd, for any $\tau\in H\{\ell\}^*$, using (\ref{gcd}), we can find integers $A_\tau,B_\tau>1$ such that
\[A_\tau \frac{\ell-1}{2}\frac{\#H}{\ell}+B_\tau (o(\tau)-1)\frac{\# H}{o(\tau)}\equiv \frac{\ell-1}{2}\frac{\#H}{o(\tau)} \pmod{o(x_\tau)}\]
and, in particular, we have
\[x_\tau^{A_\tau \frac{\ell-1}{2}\frac{\#H}{\ell}+B_\tau (o(\tau)-1)\frac{\# H}{o(\tau)}}=x_\tau^{\frac{\ell-1}{2}\frac{\#H}{o(\tau)}}.\]
Hence we conclude by Lemma \ref{cobbia311}, with $A_\tau$ and $B_\tau$ as above.
\end{proof}

\begin{theorem}\label{constructiveinclusion} 
Let $k$ be a number field and let $\G$ be a finite group of odd order satisfying the second condition of Definition \ref{verygood}. Let $H$ be an abelian group of odd order, let $\mu:\G\to\aut(H)$ be an action of $\G$ on $H$ and let $G=H\rtimes_\mu\G$. Then
\[\rt(k,G)\supseteq\rt(k,\G)^{\#H} \prod_{\ell|\#H}\prod_{\tau\in H\{\ell\}^*}W\left(k,E_{k,G,\tau}\right)^{\frac{\ell-1}{2}\frac{\#G}{o(\tau)}}.\]
More precisely, for any class $x$ in the right-hand side of the above inclusion and any positive integer $s$, there exists a tame $G$-Galois extension $K/k$ with Steinitz class $x$ and satisfying $P(s)$.
\end{theorem}

\begin{proof}
We will prove directly the last assertion of the theorem. Note that it suffices to prove it for $s$ big enough. So let $s$ be a positive integer which is a multiple of the order of $H$. Let
\[x\in \rt(k,\G)^{\#H} \prod_{\ell|\#H}\prod_{\tau\in H\{\ell\}^*}W\left(k,E_{k,G,\tau}\right)^{\frac{\ell-1}{2}\frac{\#G}{o(\tau)}}\]
and write $x=y^{\#H}z^{\#\G}$ with $y\in  \rt(k,\G)$ and $z\in \prod_{\ell|\#H}\prod_{\tau\in H\{\ell\}^*}W\left(k,E_{k,G,\tau}\right)^{\frac{\ell-1}{2}\frac{\#H}{o(\tau)}}$. By the hypothesis on $\G$, we can find a tame $\G$-Galois extension $k_1/k$ with Steinitz class $y$ and which satisfies $P(s)$.

Further, by Lemma \ref{cobbia312}, there exists an $H$-Galois extension $K/k_1$, which satisfies $P(s)$, which is Galois over $k$, with Galois group $G$, and with $\st(K/k_1)=\iota(z)$, where $\iota:\cl(k)\to \cl(k_1)$ is induced by the inclusion $k\subseteq k_1$. Then, using a well-known formula relating the Steinitz classes in towers of extensions (see for example \cite[Proposition I.1.2]{Endo}), we get
\[\st(K/k)=\st(k_1/k)^{\#H}N_{k_1/k}(\st(K/k_1))=y^{\#H}N_{k_1/k}(\iota(z))=y^{\#H}z^{\#\G}=x,\]
where $N_{k_1/k}:\cl(k_1)\to\cl(k)$ is induced by the norm from $k_1$ to $k$. 
\end{proof}

\begin{remark} The techniques described in this section can be extended to groups of even order, as in \cite{Cobbe2}. We have restricted our attention to groups of odd order only to have cleaner proofs and statements (as already noticed, there are groups of even order which are not very good). However the complications arising from extensions of even degree vanish if we assume that the class number of the base field $k$ is odd.
\end{remark}
\end{section}

\begin{section}{$A'$-groups of odd order}\label{section4}
In this section we find some slightly more explicit descriptions for the results of \cite{Cobbe1}. In particular we prove that the answer to Question \ref{question2} is affirmative for the so-called $A'$-groups of odd order, whose definition we now recall.
\begin{definition}\label{A'groups}
We define $A'$-groups inductively:
\begin{enumerate}
\item Finite abelian groups are $A'$-groups.
\item If $\G$ is an $A'$-group and $H$ is a finite abelian group of order prime to that of $\G$, then $H\rtimes_\mu \G$ is an $A'$-group, for any action $\mu:\G\to\aut(H)$ of $\G$ on $H$.
\item If $\G_1$ and $\G_2$ are $A'$-groups, then $\G_1\times \G_2$ is an $A'$-group.
\end{enumerate}
\end{definition}

\begin{proposition}\label{semidirbello}
Let $\G$ be a very good group of odd order, let $H$ be an abelian group of odd order prime to that of $\G$ and let $\mu:\G\to\aut(H)$ be an action of $\G$ on $H$. Then $G=H\rtimes_\mu \G$ is very good. In particular abelian groups are very good.
\end{proposition}

\begin{proof}
Let $\pi:G\to G/H\cong\G$ be the projection. For any $\ell|\#\G$ and any $\sigma\in G\{\ell\}$, there exists $\tilde\sigma\in\G\{\ell\}$ such that $\pi(\sigma)=\pi(\tilde\sigma)$. Further, since $\sigma$ and $\pi(\sigma)$ have the same order, it is straightforward to verify that
\[H_{k,G,\sigma}\subseteq H_{k,G/H,\pi(\sigma)}=H_{k,\G,\tilde\sigma}.\]
Hence
\[E_{k,G,\sigma}\supseteq E_{k,\G,\tilde\sigma}\]
and, by Lemma \ref{Wrovescia},
\[W(k,E_{k,G,\sigma})\subseteq W(k,E_{k,\G,\tilde\sigma}).\]
We also observe that for $\ell|\#H$, $H\{\ell\}=G\{\ell\}$. Using Theorem \ref{constructiveinclusion} (but actually also \cite[Theorem 3.19]{Cobbe1} would work in this case) and the hypothesis that $\G$ is very good, we then have:
\[\begin{split}\rt(k,G)&\supseteq\rt(k,\G)^{\#H} \prod_{\ell|\#H}\prod_{\tau\in H\{\ell\}^*}W\left(k,E_{k,G,\tau}\right)^{\frac{\ell-1}{2}\frac{\#G}{o(\tau)}}\\
&=\prod_{\ell|\#\G}\prod_{\sigma\in \G\{\ell\}^*}W(k,E_{k,\G,\sigma})^{\frac{\ell-1}{2}\frac{\#G}{o(\sigma)}}
\prod_{\ell|\#H}\prod_{\tau\in H\{\ell\}^*}W\left(k,E_{k,G,\tau}\right)^{\frac{\ell-1}{2}\frac{\#G}{o(\tau)}}\\
&\supseteq \prod_{\ell|\#G}\prod_{\sigma\in G\{\ell\}^*}W(k,E_{k,G,\sigma})^{\frac{\ell-1}{2}\frac{\#G}{o(\sigma)}}=\W(k,G).
\end{split}\]
The opposite inclusion is given by Theorem \ref{secondinclusion}. By Theorem \ref{constructiveinclusion} and the hypothesis that $\G$ is very good, also the second condition of Definition \ref{verygood} is satisfied by $G$.
\end{proof}

\begin{proposition}\label{dirbello}
Let $\G_1,\G_2$ be very good groups of odd order. Then $G=\G_1\times \G_2$ is very good.
\end{proposition}
\begin{proof}
Let $\ell$ be a prime number, let $\sigma_1\in \G_1\{\ell\}$, $\sigma_2\in \G_2\{\ell\}$ be not both trivial. Then $o(\sigma_1\sigma_2)=\max\{o(\sigma_1),o(\sigma_2)\}>1$ and let us assume $o(\sigma_1\sigma_2)=o(\sigma_1)$. Clearly $N_G(\sigma_1\sigma_2)\subseteq N_G(\sigma_1)$ and, for every $\tau\in N_G(\sigma_1\sigma_2)$, we have
\[\varphi_{\sigma_1\sigma_2}(\tau)=\varphi_{\sigma_1}(\tau).\]
Therefore
\[H_{k,G,\sigma_1\sigma_2}\subseteq H_{k,G,\sigma_1},\]
i.e.
\[E_{k,G,\sigma_1\sigma_2}\supseteq E_{k,G,\sigma_1}.\]
Using Lemma \ref{Wrovescia} and Lemma \ref{primolemmaE}, we conclude that
\[W(k,E_{k,G,\sigma_1\sigma_2})^{\frac{\ell-1}{2}\frac{\#G}{o(\sigma_1\sigma_2)}}\subseteq W(k,E_{k,G,\sigma_1})^{\frac{\ell-1}{2}\frac{\#\G_1\cdot \#\G_2}{o(\sigma_1)}}=W(k,E_{k,\G_1,\sigma_1})^{\frac{\ell-1}{2}\frac{\#\G_1}{o(\sigma_1)}\#\G_2}\]
and an analogous result holds when $o(\sigma_1\sigma_2)=o(\sigma_2)$.
Hence, using the hypothesis that $\G_1$ and $\G_2$ are very good groups of odd order,
\[\begin{split}\W(k,G)&=\prod_{\ell|\#G}\prod_{\sigma\in G\{\ell\}^*}W(k,E_{k,G,\sigma})^{\frac{\ell-1}{2}\frac{\#G}{o(\sigma)}}
\\&\subseteq \prod_{\ell|\#\G_1}\prod_{\sigma_1\in\G_1\{\ell\}^*}\!\!\! W(k,E_{k,\G_1,\sigma_1})^{\frac{\ell-1}{2}\frac{\#\G_1}{o(\sigma_1)}\#\G_2}
\prod_{\ell|\#\G_2}\prod_{\sigma_2\in\G_2\{\ell\}^*} \!\!\! W(k,E_{k,\G_2,\sigma_2})^{\frac{\ell-1}{2}\frac{\#\G_2}{o(\sigma_2)}\#\G_1}
\\&=\rt(k,\G_1)^{\#\G_2}\rt(k,\G_2)^{\#\G_1}.\end{split}\]
In particular every class in $\W(k,G)$ is of the form $x_1^{\#\G_2}x_2^{\#\G_1}$, where $x_i\in\rt(k,\G_i)$ for $i=1,2$. Now let $s$ be a positive integer: then, since $\G_1$ is very good, there exists a tame $\G_1$-Galois extension $k_1/k$ with Steinitz class $x_1$ and satisfying $P(s)$. Further let $\tilde s$ be a multiple of $s$ such that all the primes ramifying in $k_1/k$ ramify also in $k(\zeta_{\tilde s})/k$. Now, since also $\G_2$ is very good, there exists a tame $\G_2$-Galois extension $k_2/k$ with Steinitz class $x_2$ and satisfying $P(\tilde s)$. The compositum $K=k_1k_2$ is a tame $G$-Galois extension of $k$, satisfying $P(s)$. Since $k_1$ and $k_2$ are linearly disjoint over $k$, we have $d(K/k)=d(k_1/k)^{\#\G_2}d(k_2/k)^{\#\G_1}$ and hence, by \cite[Theorem 2.1]{Cobbe1}, 
\[\st(K/k)=x_1^{\#\G_2}x_2^{\#\G_1},\]
since $\#\G_1$ and $\#\G_2$ are both odd.

Therefore $\W(k,G)\subseteq\rt(k,G)$ and the opposite inclusion is given by Theorem \ref{secondinclusion}. Actually the above construction also proves the second property characterizing very good groups.
\end{proof}

Putting things together, we obtain a refinement of \cite[Theorem 3.23]{Cobbe1}.
\begin{theorem}\label{A'ok}
Every $A'$-group of odd order is very good.
\end{theorem}

\begin{proof}
Obvious, by induction, using the above results.
\end{proof}

\end{section}

\begin{section}{Groups of order dividing $\ell^4$}\label{lgruppi}\label{section5}

In this section, $\ell$ is an odd prime. We want to show that the answer to Question \ref{question2} is affirmative for all nonabelian groups of order dividing $\ell^4$ (the abelian case is covered by Theorem \ref{A'ok}).
Further the same is true for all groups of order $\ell^n$ and exponent $\ell^{n-1}$, for any positive integer $n$ (these are actually the groups studied in \cite{Cobbe3}).

\begin{lemma}\label{Elgruppi}
Let $G$ be an $\ell$-group and let $\tau\in G^*$. Set
\[e_\tau=\frac{o(\tau)}{\# (N_G(\tau)/C_G(\tau))}.\]
Then $\ell|e_\tau$ and $E_{k,G,\tau}=k(\zeta_{e_\tau})$; in particular $\zeta_\ell\in E_{k,G,\tau}$.
\end{lemma}

\begin{proof}
Since the kernel of $\varphi_\tau$ is the centralizer $C_G(\tau)$, $\varphi_\tau(N_G(\tau))$ is the subgroup of order $\# (N_G(\tau)/C_G(\tau))$ of the $\ell$-Sylow subgroup of the cyclic group $(\Z/o(\tau)\Z)^\times$. This already gives $\ell|e_\tau$ and we also have
\[\varphi_\tau(N_G(\tau))=1+\frac{o(\tau)}{\# (N_G(\tau)/C_G(\tau))}\Z\pmod{o(\tau)}=1+e_\tau\Z\pmod{o(\tau)}.\]
Therefore
\[H_{k,G,\tau}=\gal(k(\zeta_{o(\tau)})/k(\zeta_{e_\tau}))\]
and
\[E_{k,G,\tau}=k(\zeta_{e_\tau})\supseteq k(\zeta_\ell).\]
\end{proof}

\begin{proposition}[Burnside]\label{norabemax}
A group of order $\ell^n$, whose center has order $\ell^c$, contains a normal abelian subgroup of order $\ell^\nu$ where
\[\nu\geq -\frac{1}{2}+\sqrt{2n+c^2-c+\frac{1}{4}}.\]
In particular, for $n=3,4$ this becomes
\[\nu\geq n-1.\]
\end{proposition}

\begin{proof}
This result is shown in \cite[Section 4]{Burnside1913}.
\end{proof}

We shall use the above proposition for $n=3,4$. For these values of $n$ the result actually follows quite easily using standard arguments of the theory of $\ell$-groups. From now on we will use the notation $C(n)$ to denote a cyclic group of order $n$, where $n$ is a positive integer.

\begin{lemma}\label{espl}
Let $G$ be a group of exponent $\ell$ and order $\ell^n$, where $n=3,4$. Then $G$ is very good and in particular
\[\rt(k,G)=W(k,\ell)^{\frac{\ell-1}{2}\ell^{n-1}}.\]
\end{lemma}

\begin{proof}
By Proposition \ref{norabemax}, there is an exact sequence
\[1\to C(\ell)^{n-1}\to G\to C(\ell)\to 1,\]
which splits, since the exponent of $G$ is $\ell$.

Using Lemma \ref{Elgruppi}, we get $\W(k,G)=W(k,\ell)^{\frac{\ell-1}{2}\ell^{n-1}}$ and the proof of the equality $\rt(k,G)=\W(k,G)$ is straightforward, using Theorem \ref{secondinclusion} and Theorem \ref{constructiveinclusion} (which we can apply with $\G=C(\ell)$ and $H=C(\ell)^{n-1}$, since $C(\ell)$ is very good by Theorem \ref{A'ok}).

Since in particular every $x\in\rt(k,G)$ belongs to the right-hand side of the inclusion of Theorem \ref{constructiveinclusion}, it follows by Theorem \ref{constructiveinclusion} that $G$ satisfies the second condition of Definition \ref{verygood}.
\end{proof}

The above result for $n=3$ is actually proved in \cite[Corollaire 1.2]{Bruche}. The following result slightly improves \cite[Theorem 3.8]{Cobbe3}.

\begin{proposition}\label{gruppiacta}
Let $G$ be a nonabelian group of order $\ell^n$ and exponent $\ell^{n-1}$, where $n>2$ is an integer. Then $G$ is very good and
\[\rt(k,G)=W(k,\ell^{n-2})^{\frac{\ell-1}{2}\ell}.\]
\end{proposition}

\begin{proof}
As it is well known (and easily verified), $G$ can be written as
\[G=\langle\sigma,\tau:\sigma^\ell=\tau^{\ell^{n-1}}=1,\ \sigma\tau\sigma^{-1}=\tau^{1+\ell^{n-2}}\rangle.\]
Clearly $\langle\tau^\ell\rangle$ equals the center of $G$, which has therefore order $\ell^{n-2}$. Let $\rho=\tau^a\sigma^b\in G$, with $\sigma^b\neq 1$ (resp. $\sigma^b=1$), then $\tau^a\sigma^b$ (resp. $\tau$) is an element of the centralizer of $\tau^a\sigma^b$ which is not in the center of $G$. Thus, for any element $\rho\in G$, the order of $C_G(\rho)$ is strictly larger than that of the center of $G$ and hence $\#C_G(\tau)$ is a multiple of $\ell^{n-1}$. Therefore $\# (N_G(\tau)/C_G(\tau))$ divides $\ell$ and $\frac{o(\rho)}{\ell}|e_\rho$, i.e. $k(\zeta_{o(\rho)/\ell})\subseteq E_{k,G,\rho}$, by Lemma \ref{Elgruppi}.

Hence, using Lemma \ref{Wrovescia} and Lemma \ref{2.6acta}, if $\ell^2|o(\rho)$, then
\[W(k,E_{k,G,\rho})^{\frac{\ell-1}{2}\frac{\ell^n}{o(\rho)}}\subseteq W(k,o(\rho)/\ell)^{\frac{\ell-1}{2}\frac{\ell^{n}}{o(\rho)}}\subseteq W(k,\ell^{n-2})^{\frac{\ell-1}{2}\ell}.\]
If $\rho\neq 1$ and $\ell^2\nmid o(\rho)$, i.e. $o(\rho)=\ell$, then by Lemma \ref{Elgruppi} we must have $E_{k,G,\rho}=k(\zeta_\ell)$. Hence, again using Lemma \ref{Wrovescia} and Lemma \ref{2.6acta},
\[W(k,E_{k,G,\rho})^{\frac{\ell-1}{2}\frac{\ell^n}{o(\rho)}}=W(k,\ell)^{\frac{\ell-1}{2}\ell^{n-1}}\subseteq W(k,\ell^{n-1})^{\frac{\ell-1}{2}\ell}\subseteq W(k,\ell^{n-2})^{\frac{\ell-1}{2}\ell}.\]
Further, we have $N_G(\tau)=G$ and $C_G(\tau)=\langle\tau\rangle$ and therefore, by Lemma \ref{Elgruppi},
\[W(k,E_{k,G,\tau})^{\frac{\ell-1}{2}\frac{\ell^n}{o(\tau)}}=W(k,\ell^{n-2})^{\frac{\ell-1}{2}\ell}.\]
Therefore
\[\W(k,G)=\prod_{\rho\in G^*}W(k,E_{k,G,\rho})^{\frac{\ell-1}{2}\frac{\ell^n}{o(\rho)}}=W(k,E_{k,G,\tau})^{\frac{\ell-1}{2}\ell}.\]
We easily conclude by Theorems \ref{secondinclusion}, \ref{constructiveinclusion} and \ref{A'ok}, as in the proof of Lemma \ref{espl}.
\end{proof}

\begin{lemma}\label{qualiWkl}
Let $G$ be a nonabelian group of order $\ell^4$ and exponent $\ell^2$. Let $\tau\in G$ be of order $\ell^2$ and such that $\# (N_G(\tau)/C_G(\tau))$ is maximal. Then $\W(k,G)=W(k,E_{k,G,\tau})^{\frac{\ell-1}{2}\ell^2}$.
\end{lemma}

\begin{proof}
For any $\sigma\in G$ of order $\ell$, using Lemma \ref{Elgruppi}, Lemma \ref{2.6acta} and Lemma \ref{Wrovescia}, we have
\[W(k,E_{k,G,\sigma})^{\frac{\ell-1}{2}\frac{\ell^4}{o(\sigma)}}=W(k,\ell)^{\frac{\ell-1}{2}\ell^3}\subseteq W(k,\ell^2)^{\frac{\ell-1}{2}\ell^2}\subseteq W(k,E_{k,G,\tau})^{\frac{\ell-1}{2}\ell^2}.\]
Further for any $\sigma\in G$ of order $\ell^2$ we have $\# (N_G(\sigma)/C_G(\sigma))\leq\# (N_G(\tau)/C_G(\tau))$, i.e., by Lemma \ref{Elgruppi}, $E_{k,G,\sigma}\supseteq E_{k,G,\tau}$, and hence, by Lemma \ref{Wrovescia},
\[W(k,E_{k,G,\sigma})\subseteq W(k,E_{k,G,\tau}).\]
Therefore
\[\W(k,G)=\prod_{\rho\in G^*}W(k,E_{k,G,\rho})^{\frac{\ell-1}{2}\frac{\ell^4}{o(\rho)}}=W(k,E_{k,G,\tau})^{\frac{\ell-1}{2}\ell^2}.\]
\end{proof}

\begin{lemma}\label{newlemma}
Let $G$ be a (nonabelian) group of order $\ell^4$ and exponent $\ell^2$. Then there exists a normal abelian subgroup $H$ of order $\ell^3$ and an element $\tau\in G$ of order $\ell^2$ maximizing $\#(N_G(\tau)/C_G(\tau))$, such that we have one of the following possibilities:
\begin{enumerate}
\item $\tau\in H$ and $G\cong H\rtimes C(\ell)$.
\item $G\cong \langle\tau\rangle\rtimes C(\ell^2)$.
\item $\tau\not\in H$ and there exists $\sigma_1\in N_G(\tau)\setminus\langle\tau\rangle$ of order $\ell$, which acts trivially (resp. nontrivially) on $\tau$ if $N_G(\tau)=C_G(\tau)$ (resp. if $N_G(\tau)\ne C_G(\tau)$).
\end{enumerate}
\end{lemma}

\begin{proof}
Let $\sigma\in G$ be an element of order $\ell^2$, which maximizes $\#(N_G(\sigma)/C_G(\sigma))$; since, by Lemma \ref{Elgruppi}, $\ell$ divides \[e_\sigma=\frac{o(\sigma)}{\#(N_G(\sigma)/C_G(\sigma))}=\frac{\ell^2}{\#(N_G(\sigma)/C_G(\sigma))},\]
we have that $\#(N_G(\sigma)/C_G(\sigma))$ divides $\ell$. We will study separately the cases $\#C_G(\sigma)\geq \ell^3$ and $\#C_G(\sigma)=\ell^2$.

\begin{itemize}
\item If $\#C_G(\sigma)\geq \ell^3$, then we set $H$ to be a subgroup of order $\ell^3$ of $C_G(\sigma)$ containing $\sigma$. Clearly $H$ is abelian and normal, being of index $\ell$, and we have an exact sequence
\begin{equation}1\to H\to G\to C(\ell)\to 1.\label{succlemma}\end{equation}
There are three cases, corresponding to the three possibilities in the statement of the lemma.

\begin{enumerate}
\item Sequence (\ref{succlemma}) splits. Then we set $\tau=\sigma\in H$ and we have $G\cong H\rtimes C(\ell)$. 

\item Sequence (\ref{succlemma}) does not split and $\#(N_G(\sigma)/C_G(\sigma))=\ell$. In this case $N_G(\sigma)=G$, $C_G(\sigma)=H$ and we consider an element $\rho\in G\setminus H$. If $\#\langle\sigma,\rho\rangle=\ell^3$, then $\langle\sigma,\rho\rangle\cong C(\ell^2)\rtimes C(\ell)$ and there would exist an element of order $\ell$ in $N_G(\sigma)\setminus C_G(\sigma)=G\setminus H$, contradicting the assumption that (\ref{succlemma}) does not split. Therefore $\langle\sigma,\rho\rangle=G$ and, setting $\tau=\sigma$, we get $G=\langle\tau\rangle\rtimes\langle\rho\rangle$.

\item Sequence (\ref{succlemma}) does not split  and $\#(N_G(\sigma)/C_G(\sigma))=1$. In this case, by the assumption of maximality, an analogous condition holds for every element of order $\ell^2$. So we can take $\tau$ to be any element not in $H$, since it must be of order $\ell^2$ because of the nonsplitting of (\ref{succlemma}). Further let $\tilde H$ be a subgroup of order $\ell^3$ of $N_G(\tau)=C_G(\tau)$ (which has order at least $\ell^3$, since $\tau$ is of order $\ell^2$) and containing $\tau$. Clearly $\tilde H$ is abelian and we can find an element $\sigma_1$ of order $\ell$ in $\tilde H\setminus\langle\tau\rangle\subseteq N_G(\tau)\setminus\langle\tau\rangle$. In particular $\sigma_1$ acts trivially on $\tau$.
\end{enumerate}

\item If $\# C_G(\sigma)=\ell^2$, i.e. $C_G(\sigma)=\langle\sigma\rangle$, then $\# N_G(\sigma)=\ell^3$, since it must be at least $\ell^3$, being strictly larger than $\langle\sigma\rangle$, and $N_G(\sigma)$ cannot be $G$, since $\#(N_G(\sigma)/C_G(\sigma))$ divides $\ell$. By Proposition \ref{norabemax} there exists an abelian subgroup $H<G$ of order $\ell^3$. Clearly $\sigma\not\in H$, since $\# C_G(\sigma)=\ell^2$. Further $N_G(\sigma)$ is nonabelian of order $\ell^3$ and exponent $\ell^2$, so there exists $\sigma_1\in N_G(\sigma)\setminus\langle\sigma\rangle$ of order $\ell$, acting nontrivially on $\sigma$. Hence we are again in case 3, taking $\tau=\sigma$.
\end{itemize}
\end{proof}

\begin{lemma}\label{type12}
Let $G$ be of the first or second type described in Lemma \ref{newlemma}. Then $G$ is very good and
\[\rt(k,G)=W(k,E_{k,G,\tau})^{\frac{\ell-1}{2}\ell^2},\]
where $\tau\in G$ is as in Lemma \ref{newlemma}.
\end{lemma}

\begin{proof}
By Lemma \ref{qualiWkl}, we have
$\W(k,G)=W(k,E_{k,G,\tau})^{\frac{\ell-1}{2}\ell^2}$.
By Theorem \ref{constructiveinclusion}, with $\G=C(\ell)$ or $C(\ell^2)$ (note that both are very good groups by Theorem \ref{A'ok}), we obtain
\[\W(k,G)=W(k,E_{k,G,\tau})^{\frac{\ell-1}{2}\ell^2}\subseteq \rt(k,G).\]
The opposite inclusion is given by Theorem \ref{secondinclusion}. The second condition of the definition of very good groups is immediate, as in the proof of Lemma \ref{espl}.
\end{proof}

\begin{lemma}\label{type3}
Let $G$ be of the third type described in Lemma \ref{newlemma}. Then $G$ is very good and
\[\rt(k,G)=W(k,E_{k,G,\tau})^{\frac{\ell-1}{2}\ell^2},\]
where $\tau\in G$ is as in Lemma \ref{newlemma}.
\end{lemma}

\begin{proof}
By hypothesis there exists a normal abelian subgroup $H$ of $G$ of order $\ell^3$ and an element $\tau\in G\setminus H$ of order $\ell^2$ maximizing $\#(N_G(\tau)/C_G(\tau))$. We can also find an element $\sigma_1\in N_G(\tau)\setminus\langle\tau\rangle$ of order $\ell$ and acting trivially on $\tau$ if and only if $N_G(\tau)=C_G(\tau)$.
Let $\G$ be the subgroup of $G$ generated by $\sigma_1$ and $\tau$, which must be of the form $C(\ell^2)\rtimes C(\ell)$, possibly with trivial action. Clearly there exists an integer $a$ such that $\sigma_1\tau\sigma_1^{-1}=\tau^{1+a\ell}$ (since the order of the conjugation divides $\ell$). Since $\tau\not\in H$, we have that $\tau H$ generates $G/H$; in particular, $\sigma_1=\sigma\tau^b$, for some $\sigma\in H^*$ (not necessarily of order $\ell$) and $b\in\N$. We also have  \[\sigma\tau\sigma^{-1}=\tau^{1+a\ell}.\]
For later use, we note that, since $\G$ has order $\ell^3$ and exponent $\ell^2$, there must exist $c\in\{0,\,1,\,\ldots,\,\ell-1\}$ such that $\sigma^{\ell}=\tau^{c\ell}$ (an easy calculation shows that $c=-b$, but we won't need it), so that we have
\[\G=\langle \sigma,\, \tau\,:\, \tau^{\ell^2}=\sigma^\ell\tau^{-c\ell}=1, \sigma\tau\sigma^{-1}=\tau^{1+a\ell}\rangle.\]
We now consider the (abstract) group $\tilde\G$, generated by $\tilde\sigma$ and $\tilde\tau$, satisfying the same relations as $\sigma$ and $\tau$ in $\G$. We also consider the element $\tilde\sigma_1\in\tilde\G$, corresponding to $\sigma_1\in\G$.

We define an action $\mu:\tilde\G\to\aut(H)$ by setting, for any $h\in H$,
\[\begin{cases}
\mu(\tilde \tau)(h)=\tau h\tau^{-1}\\
\mu(\tilde \sigma)(h)=\sigma h\sigma^{-1}=h.
\end{cases}\]
We consider the semidirect product
\[\tilde G=H\rtimes_\mu \tilde\G.\]
For every couple $h_1,h_2$ of elements of $H$, we want to define a surjective homomorphism $\pi:\tilde G\to G$, such that $\pi(h_1\tilde\tau)$ and $\pi(h_2\tilde\tau)$ are both of order $\ell^2$ and $\ker(\pi)\cap H$ is trivial. There are two cases:
\begin{enumerate}
\item For every $h\in H$, $h\tau$ is of order $\ell^2$.

In this case we define $\pi:\tilde G\to G$ by removing tildes. Then $\pi$ is clearly a surjective homomorphism, $\pi(h\tilde\tau)=h\tau$ has order $\ell^2$ for any $h\in H$ and $\ker(\pi)$ has trivial intersection with $H$.

\item There exists an element $h_0\in H$ such that $h_0\tau$ is of order $\ell$. 

Note that in this case, if for example $h_1=h_0$, we cannot define $\pi$ simply by removing tildes because $\pi(h_0\tilde\tau)=h_0\tau$ would have order $\ell$. Then we proceed as follows: by assumption,
\[(h_0\tau)^\ell=	\left(\prod_{i=0}^{\ell-1}\tau^{i} h_0\tau^{-i}\right)\tau^\ell=1,\quad\textrm{i.e.}\quad\prod_{i=0}^{\ell-1}
(\tau^{i} h_0\tau^{-i})=\tau^{-\ell}.\]
It follows that, for any $j\in \Z$,
\[(h_0^j\tau)^\ell=\left(\prod_{i=1}^{\ell-1}\tau^{i} h_0^j\tau^{-i}\right)\tau^\ell=\left(\prod_{i=1}^{\ell-1}\tau^{i} h_0\tau^{-i}\right)^j\tau^\ell=\tau^{\ell(1-j)}.\]

Now, for any $j\in \Z$, we want to define a homomorphism $\pi_j:\tilde G \to G$ by setting
\[
\begin{cases}
\pi_j(\tilde \tau)=h_0^{j}\tau\\
\pi_j(\tilde \sigma)=\sigma^{1-j}\\
\pi_j(h)=h,\ \forall h\in H.
\end{cases}
\]
We have to verify that the definition is compatible with the relations of $\tilde G$: 
\[\begin{split}
\pi_j(\tilde \sigma)\pi_j(\tilde \tau)\pi_j(\tilde \sigma^{-1})&=\sigma^{1-j}h_0^j\tau\sigma^{-(1-j)}=h_0^j\sigma^{1-j}\tau\sigma^{-(1-j)}\\
&=h_0^j\tau^{1+a\ell(1-j)}=h_0^j\tau(h_0^j\tau)^{a\ell}=(h_0^j\tau)^{1+a\ell}\\
&=\pi_j(\tilde \tau^{1+a\ell});\\
\pi_j(\tilde \tau)\pi_j(h)\pi_j(\tilde \tau^{-1})&=h_0^j\tau h\tau^{-1}h_0^{-j}=\tau h\tau^{-1}\\&=\pi_j(\tau h\tau^{-1});\\
\pi_j(\tilde \sigma)\pi_j(h)\pi_j(\tilde \sigma^{-1})&=\sigma^{1-j}h\sigma^{-(1-j)}=h\\&=\pi_j(h);\\
\pi_j(\tilde\sigma)^\ell&=\sigma^{(1-j)\ell}=\tau^{c\ell(1-j)}=(h_0^j\tau)^{c\ell}\\&=\pi_j(\tilde\tau)^{c\ell};\\
\pi_j(\tilde\tau)^{\ell^2}&=(h_0^j\tau)^{\ell^2}=\tau^{\ell^2(1-j)}=1\\&=\pi_j(1).
\end{split}\]

So, for any $j\in \Z$, $\pi_j:\tilde G\to G$ is indeed a homomorphism, which is clearly surjective and whose kernel has trivial intersection with $H$.
Now let $h_1,h_2$ be two elements in $H$. We want to choose $j$ such that $\pi_j(h_1\tilde\tau)$ and $\pi_j(h_2\tilde\tau)$ are both of order $\ell^2$. For $i=1,2$, we can find $h_i'\in H$ such that $(h_i\tilde\tau)^\ell=h_i'\tilde\tau^\ell$. Then, for $i=1,2$,
\[\pi_j(h_i'\tilde\tau^\ell)=h_i' (h_0^j\tau)^\ell=h_i' \tau^{\ell(1-j)}.\]
Since $h_i'\tau^{\ell(1-j_1)}\neq h_i'\tau^{\ell(1-j_2)}$ if $j_1\not\equiv j_2\pmod\ell$, we have at least $\ell-1\geq 2$ choices for $j$ such that $\pi_j(h_1'\tilde\tau^\ell)\neq 1$. Among those choices, at least $\ell-2\geq 1$ give also $\pi_j(h_2'\tilde\tau^\ell)\neq 1$.
\end{enumerate}

Hence, in both cases, for every couple $h_1,h_2$ of elements of $H$, we have defined a projection $\pi:\tilde G\to G$, such that $\pi(h_1\tilde\tau)$ and $\pi(h_2\tilde\tau)$ are of order $\ell^2$ and $\ker(\pi)$ has trivial intersection with $H$. With this in mind, we start constructing the extensions needed to prove the lemma.

Let $k$ be a number field, let $x\in W(k,E_{k,G,\tau})$ and let $s$ be a positive multiple of $\ell^3$ (this hypothesis will be used later to apply Lemma \ref{cobbia311}). Since $\ell$ is an odd prime, there exist integers $A,B>1$ such that
\[\ell A+2(\ell+1)B\equiv1 \pmod{o(x)}\]
and, in particular, we have
\[x^{\ell A+2(\ell+1)B}=x.\]

By Lemma \ref{cobbia310} there exists a tame $C(\ell)$-Galois extension $\tilde k_0/k$ ramified only at principal ideals and satisfying $P(s)$. 
Since $\tilde\G=\langle\tilde\tau\rangle\rtimes\langle\tilde\sigma_1\rangle$, possibly with trivial action, and thanks to our assumption on $\sigma_1$, we have $x\in W(k,E_{k,G,\tau})=W(k,E_{k,\tilde\G,\tilde\tau})$. So we can apply Lemma \ref{cobbia311} (with $x_{\tilde\tau}=x$, $x_{\tilde\rho}=1$ for every $\tilde\rho\in\langle\tilde\tau\rangle^*\setminus\{\tilde\tau\}$, $A_{\tilde\tau}=0$ and $B_{\tilde\tau}=B$) and construct a tame $C(\ell^2)$-Galois extension $\tilde k_1/\tilde k_0$, such that $\tilde k_1/k$ is $\tilde\G$-Galois and it satisfies the following conditions.
\begin{itemize}
\item The only nonprincipal prime ideals of $k$ ramifying in $\tilde k_1/\tilde k_0$ are $\p_1,\p_2$ in the class of $x^B$ and the inertia group of every prime of $\tilde k_1$ dividing these $\p_i$ is $\langle\tilde \tau\rangle$ in $\tilde\G=\mathrm{Gal}(\tilde k_1/k)$ ($\langle\tilde \tau\rangle$ is invariant by conjugation, since it is normal in $\tilde \G$);
\item $\tilde k_1/\tilde k_0$ satisfies $P(s)$. 
\end{itemize} 
Now we want to construct an $H$-Galois extension $\tilde K/\tilde k_1$. By Lemma \ref{Wrovescia} and Lemma \ref{Elgruppi}, we have $x\in W(k,E_{k,G,\tau})\subseteq W(k,\ell)=W(k,E_{k,\tilde G,\rho})$ for every $\rho\in H$ of order $\ell$. Hence again we can use Lemma \ref{cobbia311} (with $x_{\rho_0}=x$, for a certain $\rho_0\in H^*$ of order $\ell$, $x_\rho=1$ for every $\rho\in H^*\setminus\{\rho_0\}$, $A_{\rho_0}=A$ and $B_{\rho_0}=0$) to find a tame $H$-Galois extension $\tilde K/\tilde k_1$, such that $\tilde K/k$ is $\tilde G$-Galois and satisfies the following conditions.
\begin{itemize} 
\item The only nonprincipal prime ideals of $k$ ramifying in $\tilde K/\tilde k_1$ are $\p_3,\dots,\p_{A+2}$ in the class of $x$ and the inertia group of every prime of $\tilde K$ dividing these $\p_i$ is generated by a conjugate of $\rho_0$ in $\tilde G=\mathrm{Gal}(\tilde K/k)$.
\item $\tilde K/\tilde k_1$ satisfies $P(s)$ (in particular it follows that $\{\p_1,\,\p_2\}\cap\{\p_3,\dots,\p_{A+2}\}=\varnothing$).
\end{itemize}

Let us consider two prime ideals $\tilde\P_1$ and $\tilde\P_2$ of $\tilde K$ dividing $\p_1$ and $\p_2$ respectively. Their inertia groups in the extension $\tilde K/k$ must be generated by elements of the form $h_i\tilde\tau$ for some $h_i\in H$, $i=1,2$. We consider a projection $\pi:\tilde G\to G$ for which $\pi(h_1\tilde\tau)$ and $\pi(h_2\tilde\tau)$ are both of order $\ell^2$ and such that $\ker(\pi)$ has trivial intersection with $H$. We call $K$ the fixed field of $\ker(\pi)$. By construction the ramification index of $\p_1$ and $\p_2$ in $K/k$ is $\ell^2$. Further the inertia group of the primes of $\tilde K$ dividing $\p_3,\dots,\p_{A+2}$ lies in $H$, which has trivial intersection with $\ker(\pi)$. Therefore $\p_3,\dots,\p_{A+2}$ ramify in $K/k$, with ramification index $\ell$ and all the other ramifying primes are principal. Hence the discriminant of $K/k$ is
\[d(K/k)=(\p_1\p_2)^{(\ell^2-1)\ell^2}\prod_{i=3}^{A+2}\p_i^{(\ell-1)\ell^3}I^2,\]
where $I$ is a principal ideal. Hence, by \cite[Theorem 2.1]{Cobbe1},  the Steinitz class is
\[\st(K/k)=x^{B(\ell^2-1)\ell^2+A\frac{\ell-1}{2}\ell^3}=x^{\frac{\ell-1}{2}\ell^2(2(\ell+1)B+\ell A)}=x^{\frac{\ell-1}{2}\ell^2}.\]
Hence, recalling also Lemma \ref{qualiWkl}, we have
\[\W(k,G)=W(k,E_{k,G,\tau})^{\frac{\ell-1}{2}\ell^2}\subseteq\rt(k,G),\]
We conclude, by Theorem \ref{secondinclusion}, that 
\[\rt(k,G)=\W(k,G).\]
As in the proof of Lemma \ref{espl}, the second condition of Definition \ref{verygood} holds, since the extensions we have constructed all satisfy $P(s')$ for every $s'$ dividing $s$.
\end{proof}

\begin{theorem}
Let $\ell$ be an odd prime. Then every group of order dividing $\ell^4$ is very good. In particular
\[\rt(k,G)=\W(k,G).\]
\end{theorem}

\begin{proof}
If $G$ is abelian, the theorem is a consequence of Theorem \ref{A'ok}. So suppose that $G$ is nonabelian. If $G$ is of exponent $\ell$, then the theorem follows from Lemma \ref{espl}. If $G$ is of order $\ell^n$ and exponent $\ell^{n-1}$, for $n=3,4$, then the result is given by Proposition \ref{gruppiacta}. We are therefore left with the case where $G$ has order $\ell^4$ and exponent $\ell^2$ and we conclude by Lemma \ref{newlemma}, Lemma \ref{type12} and Lemma \ref{type3}.
\end{proof}

\end{section}

\begin{section}*{Acknowledgements}
We wish to thank Maurizio Monge for sharing with us some of his insights in the theory of $\ell$-groups. Further we wish to thank the institutions which are supporting our research, i.e. the Universit\'e de Limoges and the Scuola Normale Superiore of Pisa.
\end{section}

\nocite{PARI2}

\begin{tabularx}{\textwidth}{XX}
   Alessandro Cobbe & Luca Caputo\\
   Scuola Normale Superiore & Facult\'e de Sciences et Techniques\\
   piazza Cavalieri, 7 & 123 Avenue Albert Thomas\\
   56126 Pisa & 87060 Limoges\\
   Italy & France\\
   a.cobbe@sns.it & luca.caputo@unilim.fr
\end{tabularx}

\end{document}